\documentclass[12pt]{article}
\usepackage{amssymb,amsmath,amsthm}

\newtheorem{theorem}{Theorem}%[section]

 \newcommand{\Q}{\ensuremath{\mathbb{Q}}}
 
 \newcommand{\C}{\ensuremath{\mathbb{C}}}

\newcommand{\ord}{{\rm ord}}

\begin{document}
\title{\bf On Holomorphic Artin L-functions}
\author{FLORIN NICOLAE\\
 \\Simion Stoilow Institute of Mathematics\\ of the
Romanian Academy\\ P.O.BOX 1-764\\ 
RO-014700 Bucharest\\email:florin.nicolae@imar.ro}

\maketitle

\begin{abstract}
Let $K/\Q$ be a finite Galois extension, $s_0\in \C\setminus \{1\}$,
${\it Hol}(s_0)$ the semigroup of Artin L-functions holomorphic at $s_0$.
If the Galois group is almost monomial then Artin's L-functions are holomorphic
at $s_0$ if and only if $ {\it Hol}(s_0)$ is factorial. This holds also if $s_0$ is a zero of an irreducible 
L-function of dimension $\leq 2$, without any condition on the Galois group.


\end{abstract}

\indent Let $K/\Q$ be a finite Galois extension with the Galois group $G$,
$\chi_1,\ldots,\chi_r$ the irreducible characters of $G$ with the
dimensions $d_1:=\chi_1(1),\ldots,d_r:=\chi_r(1)$, $f_1=L(s, \chi_1,
K/\Q),\ldots,f_r=L(s,
\chi_r, K/\Q)$ the corresponding Artin L-functions,
$$Ar:=\{f_1^{k_1}\cdot\ldots\cdot f_r^{k_r}\mid k_1\geq 0,\ldots,k_r\geq 0\}$$
the multiplicative semigroup of all  L-functions. Artin proved that $f_1,\ldots,f_r$ are multiplicatively
independent (\cite{Ar}, Satz 5, P. 106 ), so $Ar$ is factorial with the set of primes $\{f_1,\ldots,f_r\}$.  For $s_0\in\C,s_0\neq 1$, let $ {\it Hol}(s_0)$ be the
subsemigroup of $Ar$ consisting of the L-functions which are holomorphic at
$s_0$. Artin conjectures that every  L-function is holomorphic at  $s_0$. This is true for monomial Galois groups. If $G$ is isomorphic to $A_5$, the alternating group on five elements,  it was proved in  (\cite{Ni2}, Theorem 3) that the  L-functions are holomorphic at $s_0$ if and only if $ {\it Hol}(s_0)$ is factorial. A finite group $G$ is called {\em almost monomial} if for every 
distinct irreducible characters $\chi$ and $\psi$ of  $G$ there exist a subgroup $H$ of $G$ and a linear character $\varphi$ of $H$ such that the induced character $\varphi^G$ contains  $\chi$ and does not contain  $\psi$. Every monomial group and every quasi monomial group in the sense of (\cite {Ni1}) are almost monomial. The group $A_5$ is almost monomial. 

\begin{theorem}
If the Galois group is almost monomial, then the following assertions are equivalent:\\
1) Artin's conjecture is true: $ {\it Hol}(s_0)=Ar.$\\
2) The semigroup $ {\it Hol}(s_0)$ is factorial.
\end{theorem}
\begin{proof}  

$1)\Rightarrow 2)$: If the L-functions are holomorphic at $s_0$,
then $ {\it Hol}(s_0)=Ar$ is factorial.\\
$2)\Rightarrow 1)$: Suppose that Artin's conjecture is not true. Then there exists $1\leq k\leq r$ such that 
\begin{equation}
\ord(f_k)<0,
\end{equation}
where $\ord(f)$ denotes the order of the Artin L-function $f$ at $s_0$.
The Dedekind zeta function $\zeta_K$ of $K$ decomposes as
\begin{equation}\zeta_K=f_1^{d_1}\cdot\ldots\cdot f_r^{d_r}.\end{equation}
Since $\zeta_K $ is holomorphic in ${\mathbb C}\setminus \{1\} $ it holds that 
\begin{equation}
\ord(\zeta_K)\geq 0.
\end{equation}
From $(1)$, $(2)$ and $(3)$ it follows that 
there exists $l\in\{1,\ldots,r\}$ such that
$$\ord(f_l)>0.$$
For $j\in\{1,\ldots,r\}$ let 
\[
m_j:=\min\{ m\geq 0:\ord(f_l^m\cdot f_j)\geq 0\}.
\]
Since the L-functions $f_1,\ldots,f_r$ are multiplicatively independent the elements $f_l^{m_1}\cdot f_1,\ldots,f_l^{m_r}\cdot f_r$ are irreducible in $ {\it Hol}(s_0)$.  We have seen in \cite{Ni2}, p. 2862, that $ {\it Hol}(s_0)$ is a positive
affine semigroup which generates the free abelian group $\{f_1^{k_1}\cdot\ldots\cdot f_r^{k_r}\mid k_1\in {\mathbb Z} 0,\ldots,k_r\in {\mathbb Z} \}$ with the basis
$f_1,\ldots,f_r$. The Hilbert basis $\mathcal H$ of  $ {\it Hol}(s_0)$ is the uniquely
determined minimal system of generators of $ {\it Hol}(s_0)$, hence 
$ {\it Hol}(s_0)$ is factorial if and only if $\mathcal H$ has $r$ elements. It follows that 
$$ \mathcal H=\{f_l^{m_1}\cdot f_1,\ldots,f_l^{m_r}\cdot f_r \}.$$
From $(1)$ it follows that $m_k>0$. Since the Galois group $G$ is almost monomial there exist a subgroup $H$ of $G$ and a linear character $\psi$ of $H$ such that the induced character $\psi^G$ contains  $\chi_k$ and does not contain  $\chi_l$. By classfield theory the Artin L-function 
$L(s, \psi^G,K/\Q)$ is a Hecke L-function so it is holomorphic at $s_0$. Then $L(s, \psi^G,K/\Q)$ is a product of elements of  $\mathcal H$. Since $\psi^G$ contains  $\chi_k$ the Artin L-function  
$L(s, \psi^G,K/\Q)$ contains $L(s, \chi_k,K/\Q)=f_k$ so it contains $f_l^{m_k}\cdot f_k$ and so $f_l$ since $m_k>0$. On the other hand, since  $\psi^G$ does not contain $\chi_l$ the Artin L-function  
$L(s, \psi^G,K/\Q)$ does not contain $L(s, \chi_l,K/\Q)=f_l$, a contradiction.

\end{proof}

We don't know whether any finite group $G$ is almost monomial. We think that theorem 1 is true without no condition on the Galois group. We can prove only a partial result:

\begin{theorem}
 If $s_0$ is a zero of some $f_l$ with $d_l\leq 2$ then the following assertions are equivalent:\\
1) Artin's conjecture is true: $ {\it Hol}(s_0)=Ar.$\\
2) The semigroup $ {\it Hol}(s_0)$ is factorial.
\end{theorem}
\begin{proof}  

$1)\Rightarrow 2)$ is clear.\\
$2)\Rightarrow 1)$: Suppose that Artin's conjecture is not true.  For $j\in\{1,\ldots,r\}$ let 
\[
m_j:=\min\{ m\geq 0:\ord(f_l^m\cdot f_j)\geq 0\}.
\]
As in the proof of theorem 1 we have that the Hilbert basis of  $ {\it Hol}(s_0)$ is 
\[ \mathcal H=\{f_l^{m_1}\cdot f_1,\ldots,f_l^{m_r}\cdot f_r \}.\]
Since $\zeta_K\in  {\it Hol}(s_0)$ there exist $a_1\geq 0,\ldots,a_r\geq 0$ such that 
$$\zeta_K=\prod_{j=1}^r\left(f_l^{m_j}\cdot f_j\right)^{a_j}, $$
so 
$$f_1^{d_1}\cdot\ldots\cdot f_r^{d_r}=\prod_{j=1}^r\left(f_l^{m_j}\cdot f_j\right)^{a_j}, $$
$$d_l=a_l+\sum_{j=1,j\neq l}^rm_j a_j, $$
$$d_j=a_j,j\neq l $$
since $f_1,\ldots,f_r$ are multiplicatively
independent and $m_l=0$. Hence
\begin{equation}d_l=a_l+\sum_{j=1,j\neq l}^r m_jd_j.\end{equation}
Suppose that there exists $k\neq l$ such that $$\ord(f_k)>0.$$
For $j\in\{1,\ldots,r\}$ let
\[
n_j:=\min\{ m\geq 0:\ord(f_k^m\cdot f_j)\geq 0\}.
\]
It follows that
\[ \mathcal H=\{f_k^{n_1}\cdot f_1,\ldots,f_k^{n_r}\cdot f_r\}\]
hence
$$\{f_l^{m_1}\cdot f_1,\ldots,f_l^{m_r}\cdot f_r \}=\{f_k^{n_1}\cdot f_1,\ldots,f_k^{n_r}\cdot f_r\},$$
$$f_l^{m_j}\cdot f_j=f_k^{n_j}\cdot f_j$$ for $j\neq k,l$,  
$$f_l^{m_j}=f_k^{n_j}, $$
$$m_j=n_j=0, $$
$$\mathcal H=\{f_j:j\neq k,l \}\cup \{f_l^{m_k}\cdot f_k,f_l\}=\{f_j:j\neq k,l \}\cup \{f_k^{n_l}\cdot f_l,f_k\}, $$
$$f_k=f_l^{m_k}\cdot f_k, $$
$$f_l^{m_k}=1, $$
$$m_k=0, $$
$$n_l=0, $$
$$\mathcal H=\{f_1,\ldots,f_r\}. $$
This means that $f_1,\ldots,f_r$ are holomorphic at $s_0$, so Artin's conjecture is true, a contradiction. It follows that
\begin{equation} \ord(f_k)\leq 0\, 
\text{for every}\, k\neq l.\end{equation} Since we have supposed that Artin's conjecture is not true there exists $k\neq l$ such that
$$\ord(f_k)<0.$$
Since $m_j=0$ if $\ord(f_j)=0$ from $(4)$ and $(5)$ it follows that 
\[d_l=a_l+\sum_{j:\ord(f_j)<0}m_jd_j,\]
hence
\begin{equation} d_l\geq \sum_{j:\ord(f_j)<0}m_jd_j,\end{equation}
since $a_l\geq 0$. For any $j$ with $\ord(f_j)<0$ we have that $d_j\geq 2$, since if $d_j=1$ then by classfield theory the L-function $f_j$ is a Hecke L-function so it is holomorphic at $s_0$. Since $d_l\leq 2$ it follows from $(6)$ that there exists 
only one $k$ such that $\ord(f_k)<0$. This implies  $d_k=2$, $m_k=1$, $d_l=2$, $\mathcal H=\{f_j:j\neq k \}\cup \{f_l\cdot f_k\}$. By a result of Rhoades improving the theorem of Aramata-Brauer ( \cite {Rho}, Theorem 2, p. 359) the function $\zeta_K\cdot f_k$ is holomorphic in ${\mathbb C}\setminus \{1\}$, so 
$\zeta_K\cdot f_k\in   {\it Hol}(s_0)$. Then  $\zeta_K\cdot f_k$ is a product of elements of  $\mathcal H$:  there exist $b_1,\ldots,b_r\geq 0$ such that
\[\zeta_K\cdot f_k=(\prod_{j\neq k}f_j^{b_j})\cdot (f_l\cdot f_k)^{b_k},\]
\[(\prod_{j\neq k}f_j^{d_j})\cdot f_k^{d_k+1}=(\prod_{j\neq k}f_j^{b_j})\cdot(f_l\cdot f_k)^{b_k},\]

\[b_k=d_k+1,d_l=b_l+b_k,\]
\[d_l\geq d_k+1=3,\]
in contradiction with $d_l=2$.

\end{proof}

\end{document}